\documentclass[12pt]{amsart}
\usepackage{amssymb,amsmath,amsthm,wasysym}
\usepackage{multirow}
\newtheorem{theorem}{Theorem}[section]
\newtheorem{lemma}[theorem]{Lemma}

\theoremstyle{remark}
\newtheorem*{remark}{{\bf Remark}}

\numberwithin{equation}{section}
\newtheorem*{definition}{{\bf Definition}}

\newtheorem{claim}{Claim}

\DeclareMathOperator{\diam}{diam}
\DeclareMathOperator{\dist}{dist}

\begin{document}


\title[Bi-Lipschitz embeddability of the Grushin plane into  Euclidean space]{Bi-Lipschitz embeddability of the Grushin plane into  Euclidean space}

\author[J. Seo]{Jeehyeon Seo}
\thanks{This work was supported by US National Science Foundation Grant  DMS-0901620.}

\address
{Department of Mathematics, University of Illinois at Urbana Champaign\\
1409 West Green Street, Urbana, IL, 61801}
\email {seo6@illinois.edu}

\begin{abstract}
Many sub-Riemannian manifolds like the Heisenberg group do not admit bi-Lipschitz embedding into any Euclidean space. In contrast, the Grushin plane admits a bi-Lipschitz embedding into some Euclidean space. This is done by extending a bi-Lipschitz embedding of the singular line,
using a Whitney decomposition of its complement.
\end{abstract}


\maketitle


\section{Introduction}
Let $(X, d_X)$ and $(Y, d_Y)$ be metric spaces. A map $f:X\longrightarrow Y$ is  called bi-Lipschitz, if there is $L\geq1$ such that
\begin{equation*}
\dfrac{1}{L}\cdot d_X (x, y) \leq d_Y (f(x), f(y)) \leq L\cdot d_X(x, y)
\end{equation*}
for all $x$, $y \in X$. Many analytic properties are invariant under bi-Lipschitz mappings. Thus, if we can characterize spaces that bi-Lipschitzly embed into Euclidean space, then  we can know clearly for which  spaces the resulting analysis is genuinely new and for which ones the analysis can be seen as just Euclidean analysis on a suitable subset. Since  bi-Lipschitz mappings are related to problems about differentiability, rectifiability, etc., the study of bi-Lipschitz mappings is a good starting point to observe the class of abstract metric spaces.  Bonk, Heinonen and Saksman \cite{Bonk} showed that the problem of which metric spaces are bi-Lipschitz equivalent to finite dimensional Euclidean space, is closely related to the quasiconformal Jacobian problem, which has been studied for a long time. Also, in computer science, the bi-Lipschitz embedding problem is related to the Goemans-Linial conjecture. Recently, Cheeger and Kleiner gave a counterexample for the Goemans-Linial conjecture \cite{Kleiner}.  Indeed, they showed that the Heisenberg group does not admit a bi-Lipschitz embedding into $L^1$. 

There are well-known results concerning the bi-Lipschitz embedding problem when the target space is  Euclidean space. Pansu \cite{Pansu} proved that a version of the Rademacher theorem holds for Lipschitz maps between two Carnot groups: Every Lipschitz mapping between two Carnot groups is almost everywhere differentiable. Moreover, the differential is a Lie group homomorphism almost everywhere. Semmes \cite{Semmes} observed that Pansu's differentiability theorem can be applied to get non-embedding results. He showed that any Carnot group does not admit bi-Lipschitz embedding into Banach space with the Radon-Nikod\'{y}m property. In particular, the Heisenberg group does not embed bi-Lipschitzly into any Euclidean space. 

Cheeger \cite{Cheeger} proved remarkable extension of the Rademacher theorem in doubling $p$-Poincar\'e spaces: If $(X, d, \mu)$ is a doubling metric measure space that satisfies a $p$-Poincar\'{e} inequality for some $p \geq 1$, then $(X, d, \mu)$  has a strong measurable differentiable structure, i.e. a countable collection of coordinate patches $\{(X_{\alpha},\pi_{\alpha}) \}$ that satisfy the following conditions:
\begin{enumerate}
 \item Each $X_{\alpha}$ is a measurable subset of $X$ with positive measure and the union of  the $X_{\alpha}$'s has full measure in $X$.
 \item Each $\pi_{\alpha}$ is a $N(\alpha)$-tuple of Lipschitz functions, for some $N(\alpha) \in \mathbb{N}$, where $N(\alpha)$ is bounded from above independently of $\alpha$.
\item Given a Lipschitz function $f:X \longrightarrow\mathbb{R}$, there exists an $L^{\infty}$ function $df^{\alpha}: X_{\alpha} \longrightarrow \mathbb R^{N(\alpha)}$ so that 
\begin{equation*}
\limsup_{y\rightarrow x}\dfrac{|f(y)-f(x)-df^{\alpha}(x) \cdotp (\pi_{\alpha}(y)-\pi_{\alpha}(x))|}{d(x, y)} = 0\; \text{ for}\; \mu-\text{a.e}\; x\in X_{\alpha}.
\end{equation*}
\end{enumerate}
Also, Cheeger's work contains non-embedding result: If X admits a bi-Lipschitz embedding into some finite-dimensional Euclidean space, then at almost every point in a coordinate patch, $X_{\alpha}$, the tangent cone at that point is bi-Lipschitz equivalent to $\mathbb R^{N(\alpha)}$. We can deduce from Cheeger's theorem the known non-embedding results both for the Carnot-Carath\'eodory spaces and for Laakso space. Let us consider Cheeger's theorem on the Heisenberg group $\mathbb{H}$. The Heisenberg group can be equipped with a strong measurable differentiable structure with a single coordinate patch $(\mathbb{H}, \pi_1, \pi_2)$, where $\pi_1(x,y,t)=x$ and  $\pi_2(x,y,t)=y$.  If we assume that the Heisenberg group admits a bi-Lipschitz embedding into some finite dimensional Euclidean space, then every tangent cone at almost every point in $\mathbb {H}$ is bi-Lipschitz equivalent to $\mathbb{R}^2$, which is impossible.

In contrast to the Heisenberg group, Cheeger's theorem does not tell us anything about whether or not the Grushin plane locally embeds into some Euclidean space. The Grushin plane $\mathbb{G}$ with Lebesgue measure is a locally doubling metric measure space satisfying  locally $p$-Poincar\'{e} inequality for any $p\geq 1$ (\cite{sobolev}, \cite{Jerison}). However, self similarity of the Grushin plane allows us to consider Cheeger's coordinate patch $(K, \pi_1, \pi_2)$, where $\pi_1(x,y)=x$ and  $\pi_2(x,y)=y$ for any compact $K$ in $\mathbb{G}$. Since every tangent cone to $K$ at almost every point in $K$ is bi-Lipschitz equivalent to $\mathbb{R}^2$, we cannot conclude non-embedding result from Cheeger's theorem like the case of the Heisenberg group.

Recently, Le Donne \cite{LeDonne} has shown that each sub-Riemannian manifold embeds in some Euclidean space preserving the length of all curves. 
In his result, the metric in the target is the path metric defined by infimizing the lengths of curves contained in image set. He proves that the Grushin plane can be isometrically embedded into $\mathbb R^3$ with path metric. However, this result does not answer the question about whether or not the Grushin plane embeds  bi-Lipschitzly into Euclidean space.

In this paper, we utilize the fact that the singular line admits a bi-Lipschitz embedding into $\mathbb{R}^3$. We can extend this as a Lipschitz map using a Whitney decomposition. Also, the Grushin plane is a Riemannian manifold almost everywhere. Moreover, local parts in the Grushin plane except the singular line embed bi-Lipschitzly into $\mathbb{R}^2$ with uniform bi-Lipschitz constant. This fact enables us to consider a co-Lipschitz extension on the Grushin plane. 

\begin{theorem}
The Grushin plane equipped with the Carnot-Carath\'{e}odory distance admits a bi-Lipschitz embedding into some Euclidean space. The bi-Lipschitz constant and dimension of the Euclidean space depends only on a bi-Lipschitz constant of a bi-Lipschitz embedding of the singular line into $\mathbb{R}^3$.
\end{theorem}

\section*{Acknowledgments}
The author would like to thank Jeremy Tyson and John Mackay for their advice and careful reading.

\section{Preliminaries}
In this section, we summarize basic definitions, some properties of the Grushin plane and Assouad's embedding theorem. 

\noindent{\bf The Grushin plane}. The Grushin plane is the plane $\mathbb{R}^2$ with the horizontal distribution spanned by two vector fields
\begin {center}
$X_1= \dfrac{\partial}{\partial{x}}\;\;\;\; \text{and}\;\;\;\; X_2 = x \cdot \dfrac{\partial}{\partial{y}}$.
\end {center}

\noindent{\bf The Carnot-Carath\'{e}odory distance}. The Carnot-Carath\'{e}odory distance on the Grushin plane, $d_{\mathbb{G}}(p, q)$, is defined by the infimum of the lengths of horizontal curves joining $p$ to $q$ for any two points in $\mathbb{G}$. The horizontal curve means that curve is tangent to the horizontal distribution. Since $[X_1, X_2] =\dfrac{\partial}{\partial{y}}$, the Carnot-Carath\'{e}odory distance between two points is finite by Chow's theorem.

\noindent{\bf Distance estimates}. Since the Grushin plane has dilation $\delta_{\lambda}(x, y)=(\lambda x, {\lambda}^2y)$, we have the following distance estimate along the singular line $x=0$:
\begin{equation}
 d_G((0,y_1), (0,y_2)) =\sqrt{|y_1-y_2|} d_G((0,0),(0,1)).
\end{equation}
Also, the associated sub-Riemannian metric is in fact Riemannian on the complement of the singular line $x=0$, and is equal to $ ds^2= dx^2+x^{-2}dy^2$.
Thus, we can calculate following distance estimates for any two points $p=(x_1,y_1)$ and $q=(x_2, y_2)$ on the complement of the singular line $x=0$.
\begin{equation}
\frac{1}{2}\left(|x_1-x_2|+\dfrac{|y_1-y_2|}{ \sqrt{{\min(|x_1|,|x_2|)}^2+4|y_1-y_2|}}\right)\leq d_G(p,q) \leq  4(|x_1-x_2|+\sqrt{|y_1-y_2|}). 
\end{equation}

The upper bound in (2.2)  comes from triangle inequality \cite{Bell}. To get the lower bound in (2.2), we use Riemannian metric. Let $\gamma(t) = (X(t), Y(t))$ be a parametrized horizontal curve joining $p$ to $q$ where $t \in [0,1]$. Then,
\begin{equation*}
 \text{length}(\gamma)=\int_{0}^{1} \sqrt {{\acute{X(t)}}^2+\dfrac{{\acute{Y(t)}}^2}{X(t)^2}}dt.
\end{equation*}
Thus, $(*)$ length($\gamma) \geq |x_2-x_1|$. If there exists $K$ such that $|X(t)| \leq K$ for
all $t \in [0,1]$, then length ($\gamma) \geq K^{-1}|y_2-y_2|$. Otherwise, $|X(t_{0})|> K$ for some $t_{0} \in [0,1]$ and length($\gamma$) $\geq$ length($\widetilde{\gamma}) \geq \max \{ |X(t_{0})-x_1|, |X(t_{0})-x_2|\} \geq K- \min\{|x_1|, |x_2|\}$ where $\widetilde{\gamma}$ is a subcurves of $\gamma$ joining $p$ to $(X(t_{0}), Y(t_{0}))$ or $q$ to $(X(t_{0}), Y(t_{0}))$. Thus,
\begin{equation*}
(**)\;\text{length}(\gamma) \geq \sup_{K> \max\{|x_1|, |x_2|\}} \min \{ K- \min\{|x_1|, |x_2|\}, \; K^{-1} |y_2-y_1|\}. 
\end{equation*}

When we choose $K=\sqrt{{\min(|x_1|, |x_2|)}^2+4|y_1-y_2|}$ and average $(*)$ and $(**)$, then we get the preceding distance estimates (2.2).
 
\begin{theorem}[Assouad \cite{Assouad}]
Each snowflaked version of a doubling metric space admits a bi-Lipschitz embedding in some Euclidean space. Moreover, when there is a bi-Lipschitz embedding from $(\mathbb{R}, {d_E}^p)$ into $(\mathbb {R}^n, d_E)$ where $0 < p <1$, the minimal dimension of Euclidean space is the smallest integer strictly greater than $ \frac{1}{p} $.
\end{theorem}

Assouad embedding theorem tells us that the singular line $ A=\{x=0\}$ of the Grushin plane $\mathbb{G}$ equipped with Carnot-Carath\'{e}dory distance admits a bi-Lipschitz embedding $f$ into $\mathbb{R}^3$ because the Carnot-Carath\'{e}dory metric on the singular line is comparable to the square root of Euclidean metric (see (2.1)). In this paper, we denote by $L$ a bi-Lipschitz constant of a bi-Lipschitz embedding $f: (A, d_{\mathbb G}) \longrightarrow (\mathbb{R}^3, d_E)$. Thus,
\begin{equation}
\dfrac{1}{L} \cdot {d_{\mathbb{G}}(p, q)} \leq |f(p)-f(q)| \leq L \cdot {d_\mathbb{G}(p, q)},
\end{equation}
where $p$ and $q$ are on the $y$-axis.

Now, we observe a Whitney decomposition of the complement of the singular line in $\mathbb{G}$ and an associated Lipschitz partition of unity that enable us to extend $f$ to a global Lipschitz map on the Grushin plane. 

\begin{lemma}[A Whitney decomposition $W_{\Omega}$ of $\Omega=\mathbb{G}\setminus A$]
Let $A$ be the y-axis in $\mathbb{G}$. Then its complement $\Omega$ is the union of a sequence of boxes $Q_{n,k}=[\frac{1}{2^n},\frac{1}{2^{n-1}}] \times [\frac{k}{2^{2n}},\frac{k+1}{2^{2n}}]$ and $Q^{\prime}_{n,k}=[\text{-}\frac{1}{2^{n-1}},\text{-}\frac{1}{2^{n}}] \times [\frac{k}{2^{2n}},\frac{k+1}{2^{2n}}]$ for any integers ${n}$ and ${k}$, whose interiors are mutually disjoint and whose diameters are approximately proportional to their distances from $A$.
More precisely,
\begin{enumerate}
\item $\Omega = \cup_{n \in \mathbb{Z}}(W_{\Omega}(n) \cup W_{\Omega}(\overline{n}))$ where $W_{\Omega}(n)$ =$\{Q_{n,k}\}_{k \in \mathbb Z}$ and $W_{\Omega}(\overline{n})$=$\{Q^{\prime}_{n,k}\}_{k \in \mathbb Z}$
\item The interiors of any two boxes are mutually disjoint.
\item $\dist_{\mathbb{G}}(Q, A) \leq \diam_{\mathbb{G}}(Q) \leq 8 \dist_{\mathbb{G}}(Q, A)$.
\end{enumerate}
\end{lemma}
We can easily get the constants in (3) from the distance estimates (2.2) of the Grushin plane. We emphasize that $\dist_{\mathbb G}$ and  $\diam_{\mathbb G}$ in this paper are distance and diameter with respect to the Carnot-Carath\'{e}odory distance. For $Q \in W_{\Omega}(n)$, parent  of $Q$ is the unique element $R$ in $W_{\Omega}(n-1)$  such that $Q \cap R$ contain an edge and ancestor is the parent of the parent. Also, child of $Q$ is the unique element $R$ in $W_{\Omega}(n+1)$  such that $Q \cap R$ contain an edge and descendant is the child of child. 
\begin{lemma}[A Lipschitz partition of unity]
For each $Q \in W_{\Omega}$, we denote by ${Q}^*$ the set of Whitney boxes which touch $Q$ and $Q^{**}$ the set of  Whitney boxes which touch $Q^*$.
Then, we can associate a Lipschitz partition of unity $\{ {\varphi}_{Q^*} \}$ with the following properties:
\begin{enumerate}
\item $0 \leq  {\varphi}_{Q^*} \leq 1 $,
\item ${\varphi}_{Q^*}|_{Q^*} \geq \frac{1}{C_1} >0 $ and  ${\varphi}_{Q^*}|_{(Q^{**})^c}= 0$,
\item ${\varphi}_{Q^*}$ is Lipschitz with constant $ C_2/ \diam_{\mathbb{G}}(Q)$
\item For every $p \in \Omega$, we have ${\varphi}_{Q^*}(p) \neq 0 $ for at most $C_3$ boxes $Q$,
\item $\sum_{Q\in W_{\Omega} }{\varphi}_{Q^*} =1.$
\end{enumerate}
Here, $C_1, C_2 $ and $C_3$ denote uniformly fixed constants independent of the choice of Whitney box. 
\end{lemma}
\begin{remark}
In this paper, we need a partition of unity $\varphi_{Q^*}$ that is bigger than $\frac{1}{C_1}$ on each $Q^*$ so that $\widetilde h_{Q^*}={\widehat h}\cdotp { \varphi_{Q^*}}$ in Lemma 3.6 is  bi-Lipschitz on $Q^*$, Lipschitz on $\Omega$ and supported on $Q^{**}$. 
\end{remark}

\section{Proof of Theorem 1}
We break the proof of the main theorem into two parts. In the first step, we extend a (bi)-Lipschitz map $f$ on the singular line to a global Lipschitz map $g$ on the Grushin plane using Whitney Lipschitz extension theorem. Next, we construct a co-Lipschitz map using a local and large scale argument in the sense of Whitney distance.

\subsection{Lipschitz extension Theorem on the Grushin plane}
\begin{lemma}[Lipschitz Extension Theorem \cite{Stein} ]
Let $A$ be the singular line of the Grushin plane $\mathbb{G}$ and $f$ be the bi-Lipschitz embedding of $(A, d_\mathbb{G})$ into $\mathbb{R}^3$. 
Suppose that $W_{\Omega}$ is the Whitney decomposition of $\Omega = \mathbb{G} \setminus A$ and $\{ \varphi_{Q^*} \}$ is the Lipschitz partition of unity. Then, 
\begin{equation}
g(p)= 
\begin{cases}
\sum_{Q \in W_{\Omega}}f(z_{Q}) \varphi_{Q^*}(p), & \text\;{for}\; p \in \Omega,\; z_Q\in A \;\text{ such that}\; \dist_{\mathbb{G}}(A, Q)=\dist_{\mathbb{G}}(z_Q, Q);\\
f(p), & \text{for}\; p \in A.
\end{cases}
\end{equation}
is a Lipschitz extension of $f$ on the Grushin plane.
\end{lemma}
We now claim that the following hold: 
\begin{claim}
If $p \in \Omega$, then $|X_ig(p)| \leq C \cdot L$ for $i=1,\;2$ where $C$ is a constant. 
\end{claim}
\begin{proof}[Proof of the claim]
Suppose that $p\in Q \in W_{\Omega}$ and pick $z_Q \in A$ such that $\dist_{\mathbb G}(A, Q)=\dist_{\mathbb G}(z_Q, Q)$, then for i=1,2, 
\begin{equation}
 X_ig(p) =\sum_{Q\in W_{\Omega}}f(z_Q)X_i {\varphi_{Q^*}}(p)\\
         =\sum_{Q\in W_{\Omega}}(f(z_Q)-f(z))X_i{\varphi_{Q^*}}(p) \; 
\end{equation}
because  $\sum  {\varphi_{Q^*}}(p) =1$. 

Now, let us consider boxes $R \in W_{\Omega}$ such that $p \in Q \in R^{**}$. Then, we observe followings from a Lipschitz partition of unity and a Whitney decomposition.
\begin{itemize}
\item the number of such Whitney box $R$ is at most $C_3$.
\item $d_{\mathbb G}(z_R, z_Q) \leq C d_{\mathbb G}(z_Q, p)$ where $a=(32\cdot 8 +1)$.
\item $ \frac {1}{8} \diam_{\mathbb G}(Q) \leq d_{\mathbb G}(p,z_Q) \leq 2\diam_{\mathbb G}(Q)$.
\item $\frac{1}{16}\diam_{\mathbb G}(R) \leq \diam_{\mathbb G}(Q)\leq 16 \diam_{\mathbb G}(R)$.
\end{itemize}
Therefore, we arrive at
\begin{align*}
|X_ig(p)|& \leq C_2 \cdot L \sum_{R\in W_{\Omega} \text{and} \; p \in   R^{**}} d_{\mathbb G}(z_R, z_Q){\diam_{\mathbb G}(R)}^{-1}  \\ 
         & \leq C_2\cdot C_3\cdot L \cdot 34\diam_{\mathbb G}(Q){\diam_{\mathbb G}(R)}^{-1}\\
         & \leq C\cdot L.                       
\end{align*}
because $f$ is bi-Lipschitz with constant $L$ and $|X_i {\varphi_{R^*}}(p)  | \leq C_2 {\diam_G(R)}^{-1}$ and  $C $ is the constant $ C_2\cdot C_3\cdot 34 \cdot 16$.
\end{proof}
\begin{proof}[Proof of Lemma 3.1]Now, if $p \in \Omega$  and $q \in A$, then
\begin{align*}
|g(p)-g(q)| &= |\sum _{Q \in W_{\Omega}} f(z_Q){\varphi}_{Q^*}(p)-f(q)|\\
            &\leq  \sum _{Q \in W_{\Omega}}|f(z_Q)-f(q)||{\varphi}_{Q^*}(p)|\\
            &\leq \sup_{d_{\mathbb G}(q,z_Q) \leq ad_{\mathbb G}(p,q)}|f(z_Q)-f(q)| \\
            &\leq a\cdot L\cdot d_{\mathbb G}(p,q).
\end{align*}
In the case of  $p, \;q \in \Omega$, let us consider $\gamma$ a minimizing geodesic joining $p$ to $q$. If $\gamma$ does not touch the singular line, then, from the claim, we have
\begin{align*} 
|g(p)-g(q)| &\leq d_{\mathbb G}(p, q) \cdot  \sup_{p^{\prime} \in \gamma}|( X_1g(p^{\prime}),  X_2g(p^{\prime}))|\\
                             &\leq C\cdot L\cdot d_{\mathbb G}(p,q).
\end{align*}
Otherwise, we can find a point $p^{\prime} \in A $ on  $\gamma$ such that $d_{\mathbb G}(p^{\prime}, p) < d_{\mathbb G}(p,q)$ and $d_{\mathbb G}(p^{\prime}, q) < d_{\mathbb G}(p,q)$. Then, we can deduce that
\begin{align*}
|g(p)-g(q)| &\leq|g(p)-g(p^{\prime})|+|g(p^{\prime})-g(q)|\\
            &\leq a\cdot L\cdot d_{\mathbb G}(p, p^{\prime})+ a\cdot L\cdot d_{\mathbb G}(q, p^{\prime})\\
            &\leq 2\cdot a \cdot L\cdot d_{\mathbb G}(p,q).   
\end{align*}
\end{proof}

We remark that the map $g$ is not globally co-Lipschitz on $\Omega$. Actually, $g$ is a co-Lipschitz map for any two points $p$, $q$ which are in specific relation to each other and to $A$. See Lemma 3.2 (1). Thus, we introduce a relative distance map $d_W$ which is the key tool to construct a global co-Lipschitz map on $\Omega$. In this paper, W-locally co-Lipschitz means that it is co-Lipschitz for any two points in Whitney boxes whose Whitney distance is less than $2000L^2$ and W-large scalely co-Lipschitz means that it is co-Lipschitz for any two points in Whitney boxes whose Whitney distance is greater than $2000L^2$.

\begin{definition}
The Whitney distance map $d_W$ on $W_{\Omega} \times W_{\Omega}$ is
defined by
\begin{equation*}
d_{W}(Q, R) = \dfrac{\dist_{\mathbb G}(Q, R)}{\min(\diam_{\mathbb G}(Q), \diam_{\mathbb G}(R))}.
\end{equation*}
\end{definition}

By using this definition, we break $\Omega$ into two parts and construct co-Lipschitz map on those parts. More precisely, we make co-Lipschitz maps on one part, (1) for any $p \in Q$ and $q\in R$ such that $d_W(Q, R) > 2000L^2$ and on another part, (2) for any $p \in Q$ and $q\in R$ such that $d_W(Q, R) \leq 2000L^2$.

\subsection{W-large scale co-Lipschitz and global Lipschitz map on $\Omega$}
First, we construct a $W$-large scale co-Lipschitz and global Lipschitz map on $\Omega$. Roughly speaking, the following Lemma 3.2 means that the Whitney Lipschitz extension map guarantees W-large scalely co-Lipschitz when two points are relatively close to the singular line $A$. Also, when one point is close to the singular line and another point is far away from the singular line,  distance map from $A$ is a W-large scale co-Lipschitz map.
\begin{lemma}
For any $p \in Q$ and $q \in R$ where $d_W(Q, R)>2000L^2$, the  Lipschitz extension map $g$ and $\dist_{\mathbb G}( \cdot, A)$ guarantee W-large scale co-Lipschitz bounds. More precisely, 
\begin{enumerate}
 \item If $\dfrac {\dist_{\mathbb G}(Q, R)}{\max(\diam_{\mathbb G}(Q),\diam_{\mathbb G}(R))} \geq {1000L^2}$,\;then $|g(p)-g(q)| \geq C(L) \cdot d_{\mathbb G}(p, q)$.
 \item If $\dfrac{\dist_{\mathbb G}(Q, R)} {\max(\diam_{\mathbb G}(Q),\diam_{\mathbb G}(R))}  \leq {1000L^2}$,\;then $|\dist_{\mathbb G}(p, A)-\dist_{\mathbb G}(q, A)| \geq C(L) \cdot d_{\mathbb G}(p, q)$.
\end{enumerate}
\end{lemma}

\begin{proof}
We can assume that $\diam_{\mathbb G}(R) \geq \diam_{\mathbb G}(Q)$ without loss of generality.
To prove the first case, choose $z, z^{\prime} \in A$ such that $ \dist_{\mathbb G}(A, Q) =\dist_{\mathbb G}(z,Q)$ and $ \dist_{\mathbb G}(A, R) =\dist_{\mathbb G}(z^{\prime}, R)$, then, by triangle inequality
\begin{equation*}
|g(p)-g(q)|\geq |g(z)-g(z^{\prime})|-|g(z_Q)-g(p)|-|g(z^{\prime})-g(q)| 
\end{equation*}
Now, $d_{\mathbb G}(p, z) \leq 2\diam_{\mathbb G}(Q) \leq \frac{1}{1000L^2}\dist_{\mathbb G}(Q, R) \leq \frac{1}{1000L^2} \cdot d_{\mathbb G}(p,q)$ asserts the following:
\begin{equation*}
|g(z)-g(p)| \leq \sup_{d_{\mathbb G}(z,z_Q)< ad_{\mathbb G}(z,p)} |f(z)-f(z_Q)|
            \leq a\cdot L \cdot d_{\mathbb G}(z, p)
            \leq \frac{a}{1000L}\cdot d_{\mathbb G}(p,q). 
\end{equation*}
Similarly, we can say that
\begin{equation*}
|g(q)-g(z^{\prime})| \leq a\cdot L \cdot d_{\mathbb G}(z^{\prime}, q) \leq \frac{a}{500L}\cdot d_{\mathbb G}(p,q). 
\end{equation*}
because $d_{\mathbb G}(q, z^{\prime}) \leq 2\diam_{\mathbb G}(R) \leq \frac{1}{500L^2}\cdot d_{\mathbb G}(p,q)$. 

Moreover, the fact that $g|_A = f$ is bi-Lipschitz implies
\begin{equation*}
|g(z)-g(z^{\prime})| \geq \frac{1}{L}\cdot d_{\mathbb G}(z,z^{\prime}) \geq \frac{1}{L}\cdot d_{\mathbb G}(p,q). 
\end{equation*}
Therefore, $|g(p)-g(q)| \geq C(L) \cdot d_{\mathbb G}(p,q)$.

For the second case, we observe that $\diam_{\mathbb G}(Q) \leq \frac{1}{1000L^2} \cdot \diam_{\mathbb G}(R)$ from   $d_W(Q, R) > 2000L^2$. Also, 
$\dist_{\mathbb G}(Q, R) \geq \dist_{\mathbb G}(R, A)-\dist_{\mathbb G}(Q, A) \geq \frac{1}{16}\diam_{\mathbb G}(R)$ implies that 
\begin{equation*}
d_{\mathbb G}(p,q) \leq \frac{1}{2000L^2}\cdot \diam_{\mathbb G}(R)+ \dist_{\mathbb G}(Q,R)+\diam_{\mathbb G}(R) \leq (\frac{16}{1000L^2}+17)\cdot \dist_{\mathbb G}(Q, R).
\end{equation*}
Thus, we have the following:
\begin{align*}
|\dist_{\mathbb G}(p, A)-\dist_{\mathbb G}(q, A)| & \geq |\dist_{\mathbb G}(R,A)- (\dist_{\mathbb G}(Q, A) +\diam_{\mathbb G}(Q))|\\
                               & \geq \frac{1}{8}\diam_{\mathbb G}(R)-2\diam_{\mathbb G}(Q)\\
                               & \geq \frac{1}{16} \diam_{\mathbb G}(R)\\
                               &\succsim \frac{1}{1000L^2}\cdot \dist_{\mathbb G}(Q, R)\\
                               & \succsim \frac{1}{1000L^2}\cdot d_{\mathbb G}(p,q).
\end{align*}
\end{proof}

Since we already showed that $g$ is globally Lipschitz on $\mathbb G$ and $\dist(\cdotp, A)$ is obviously globally Lipschitz on $\mathbb G$, $g \times \dist(\cdotp, A)$ is the desired $W$-large scale co-Lipschitz and global Lipschitz map on $\mathbb G$.

\subsection{W-local co-Lipschitz and global Lipschitz map on $\Omega$}
It remains to construct a $W$-local co-Lipschitz map. The dimension 4 is not enough to construct a co-Lipschitz map. Thus, we use a coloring map that gives additional dimension of the Euclidean space.
\begin{lemma}
 For any $Q \in W_{\Omega}$, the number of boxes $R$ in $W_{\Omega}$ such that $d_{W}(Q, R) < 2000L^2$ is finite.
\end{lemma}

\begin{proof}
For any fixed $Q \in W_{\Omega}(n_0) \subset W_{\Omega}$,
the number of Whitney boxes within the W- ball centered at $Q$ of radius $2000L^2$ is 
the sum of the numbers of Whitney boxes in $W_{\Omega}(n)$ and $W_{\Omega}(\overline{n})$ for all $n \in \mathbb Z$ that are within the W-ball centered at $Q$ of radius $2000L^2$.
That means
\begin{equation*}
\sharp \{R\in W_{\Omega}\;|\;d_W(Q, R)<2000L^2\} 
= \sum_{n\in \mathbb{Z}} \sharp\{R \in W_{\Omega}(n) \cup W_{\Omega}(\overline{n}) \;|\;d_W(Q, R)<2000L^2\}.
\end{equation*}

We can simply calculate from the third property of the Whitney decomposition that the numbers of maximal ancestors and descendants of $Q$ and reflection of $Q$ with respect to $y$-axis are finite constants depending only on $L$. Also, we can count the numbers of Whitney boxes in each $W_{\Omega}(n)$ and $W_{\Omega}(\overline{n})$ within $W$-$2000L^2$ ball centered at $Q$ from the distance estimate (2.2). They are finite constants depending only on $L$. Therefore, there are a finite number of Whitney boxes within the $W$-$2000L^2$ ball centered at $Q$. 

Moreover, the number of Whitney boxes within $W$-$2000L^2$ ball depends only $L$. Actually, the number of such boxes is approximately $L^4(\log_2 L)^2$.
\end{proof}

From the preceding lemma, for any  $Q \in W_{\Omega}$, $\sharp \{R \in W_{\Omega} \; | \;d_{W}(Q, R) < 2000L^2\}$ ball is at most a finite number $m(L)$. Therefore, there exists a coloring map such that any two boxes within $W$-distance $2000L^2$ have different colors.
\begin{lemma}
There exists a coloring map
\begin{equation*}
K : W_{\Omega} \longrightarrow \{1,2,3,\ldots, M \}\; \text{where}\; M \geq m(m-1)
\end{equation*}
such that any two boxes within $W$-distance ball of radius $2000L^2$ have different colors. In other words,
 if $R'$, $R''$ have $d_{W}(R', R'') <2000L^2$, then $K(R') \neq K(R'')$.
\end{lemma}

\begin{proof}
We apply Zorn's lemma. Let us consider the partially ordered set $(\mathbf{P}, \leqslant)$ where
$\mathbf{P}$ is the collection of maps $k$ defined on $S \subset W_{\Omega}$ to  $ \{1,2, \ldots ,M\}$ so that all boxes $R \in S $ within $W$-$2000L^2$ ball have all different colors.  $(k, S) \leqslant  (k^{\prime}, S^{\prime})$ means $ k^{\prime}$ is a extension of $k$. $(S\subset S^{\prime} \in \mathbf{P} \; \text{and} \; k^{\prime} |_S = k|_S )$. 

Since $\bigcup S$ is the upper bound in $(\mathbf{P}, \leqslant)$, there is a maximal element $ \widehat{k} $. If domain of $ \widehat{k}$ is $W_{\Omega}$, then we can set $K$ as $\widehat{k}$. Otherwise, we can take $Q^{\prime} \in W_{\Omega} \setminus \text{domain}( \widehat{k} )$. Then, $Q^{\prime}$ is involved in $W$-$2000L^2$ ball centered at $Q^{\prime}$ and each Whitney boxes $R$ in $d_W(Q^{\prime}, R) <2000L^2$ is involved in $W$-$2000L^2$ ball centered at $R$. From Lemma 3.3, the total number of colors seen is less than or equal to  $m(m-1)$ that is finite. Therefore, when we choose $ M \geq m(m-1)$, then we can take $\widehat{k}(Q^{\prime})$ as a different color.
\end{proof}

Now, we are ready to make a $W$-local co-Lipschitz and global Lipschitz map on $\Omega$. We observe bi-Lipschitz embeddings on local patches into $\mathbb {R}^2$ with uniform bi-Lipschitz constant. Next, we will put together all patches to make a $W$-local co-Lipschitz map by assigning different colors to each Whitney boxes. Here, $\{e_1,\; e_2,\; \ldots,\; e_M \}$ is an orthonormal basis of $\mathbb R^{M}$.
 
\begin{lemma}
When we define the map $h$ from $\Omega$ into $\mathbb {R}^2$ as the following:
\begin{equation*}
  h(x,y) =(x, {\diam_{\mathbb G}(Q)}^{-1}\cdot y)
\end{equation*}
then, the restriction map on each $Q^*$,  $h|_{Q^*}$, is bi-Lipschitz from $Q^*$ into $\mathbb {R}^2$ with uniform bi-Lipschitz constant. Also, we have a map $\widehat h$ from $\Omega$ to $\mathbb {R}^2$ so that there exists $M_1 >1$ which is independent of Whitney boxes such that the restriction on each $Q^*$, $\widehat h|_{Q^*}(Q^*) \subset B(0, M_1\diam_{\mathbb G}(Q)) \setminus B(0, \frac{1}{M_1} \diam_{\mathbb G}(Q))$.
\end{lemma}

\begin{proof}
 Let $ Q \in W_{\Omega}(n)$ be $[ \frac{1}{2^n},\frac{1}{2^{n-1}}] \times [\frac{k}{2^{2n}},\frac{k+1}{2^{2n}} ]$ for some $k$. Then $Q^* \subset [ \frac{1}{2^{n+1}},\frac{1}{2^{n-1}}] \times [\frac{j}{2^{2n+2}},\frac{j+32}{2^{2n+2}} ]$ for some $j$ such that $[k, k+1] \subset [j, j+32]$.
For any $p=(x_1, y_1)=(\frac{1+t}{2^{n+1}}, \frac{s}{2^{2n+2}})$ and $q=(x_2, y_2)$=$(\frac{1+t^{\prime}}{2^{n+1}}, \frac{s^{\prime}}{2^{2n+2}}) \in Q^*$ where $0 \leq t,\; t^{\prime} \leq 3$  and $j \leq  s, \;s^{\prime} \leq j+32 $ , we obtain 
\begin{align*}
 |h(p)-h(q)| &\approx |x_1 -x_2|+{\diam_{\mathbb G}(Q)}^{-1}|y_1-y_2|\\
                     &\approx \frac{|t-t^{\prime}|}{2^{n+1}} + {2^n} \frac{|s-s^{\prime}| }{2^{2n+2}} 
\end{align*}
From the distance estimate (2.2) of the Grushin plane, we have the following:
\begin{equation}
d_{\mathbb G}(p, q) \leq |h(p)-h(q)|  \leq 24 d_{\mathbb G}(p,q).
\end{equation}
Thus, for each $Q^*$, when we consider $\widehat h$ as composition of $h$ and an isometric translation map, $h'(x,y)=(x,y\text{-}\frac{k-1}{2^n} )$, then we can conclude that
\begin{equation*}
\widehat h|_{Q^*}(Q^*) \subset B(0, M_1\diam_{\mathbb G}(Q)) \setminus B(0, \frac{1}{M_1} \diam_{\mathbb G}(Q)) 
\end{equation*}
\end{proof} 

\begin{lemma}
The following map $H$ from $\Omega$ into $(\mathbb {R}^2)^M$ given by
\begin{equation}
H(p) = \sum_{Q \in W_{\Omega}} \widetilde{h}_{Q^*}(p) \otimes e_{K(Q)},
\end{equation}
is a global Lipschitz and W-local co-Lipschitz map where $\widetilde h_{Q^*} $ is defined by ${\widehat h}\cdot { \varphi_{Q^*}}$. The ($W$-local) bi-Lipschitz constant depends only on $L$.  That is, 
\begin{equation*}
 |H(p)-H(q)| \geq  C(L) \cdot d_{\mathbb G}(p, q) 
\end{equation*}
 for  any  $p \in Q$ , $q \in R $ where $d_W(Q,R)< 2000L^2$.
\end{lemma}

\begin{proof}
Since $\widetilde h_{Q^*}$ is bi-Lipschitz on $Q^*$ with uniform bi-Lipschitz constant, Lipschitz on $\Omega$ and supported on $Q^{**}$, the map $H$ is the $C_3$ finite sum of Lipschitz maps on $\Omega$. Thus, it is  Lipschitz on $\Omega$. Now, we will show that $H$ is  $W$-locally co-Lipschitz according to positions of two points $p$ and $q$.

If $p,\; q \in Q^*$, then $\widetilde h _{Q^*}$ is bi-Lipschitz on $Q^*$ and $Q$ is the Whitney box that shares same color at $p$ and $q$.
\begin{equation*}
| H(p)-H(q)| \geq | \widetilde h_{Q^*}(p)- \widetilde h_{Q^*}(q)|
    \geq \frac{1}{C_1}|\widehat h_{Q^*}(p)-\widehat h_{Q^*}(q)| \\
               \geq \frac{1}{C_1}d_{\mathbb G}(p, q)\; \text{from (3.3)}.
\end{equation*}

If $p \in Q,\; q\notin Q^{**}$, then $\widetilde h_{Q^*}(q) = 0$. Thus,
\begin{align*}
 | H(p)-H(q)| &\geq | \widetilde h_{Q^*}(p)- \widetilde h_{Q^*}(q)|=|\widetilde h_{Q^*}(p)|\\
              &\geq \frac{1}{M_1}\diam_{\mathbb G}(Q) \\
              &\geq \frac{1}{M_1}\min(\diam_{\mathbb G}(Q), \diam_{\mathbb G}(R))\\
              &\geq \frac{1}{M_1}\frac{1}{32(2000L^2+1)}d_{\mathbb G}(p, q).
\end{align*}
The last inequality is derived from a Whitney decomposition of the Grushin plane and $d_W(Q, R) < 2000L^2$. More precisely, we have
\begin{itemize}
\item  $ d_{\mathbb G}(p, q) \leq \dist_{\mathbb G}(p, A)+\dist_{\mathbb G}(q, A) \leq  4\max\{\diam_{\mathbb G}(Q), \diam_{\mathbb G}(R)\}$ 
\item  $2000L^2 \min\{\diam_{\mathbb G}(Q), \diam_{\mathbb G}(R)\}  > \dist_{\mathbb G}(Q, R) > |\dist_{\mathbb G}(R, A)- \dist_{\mathbb G}(Q, A)| \\
             \geq \frac{1}{8}\max\{ \diam _{\mathbb G}(Q), \diam_{\mathbb G}(R)\}-\min\{\diam_{\mathbb G}(Q), \diam_{\mathbb G}(R)\}$
\end{itemize}
Thus, we have 
\begin{equation*}
\min \{\diam_{\mathbb G}(Q), \diam_{\mathbb G}(R) \} 
\geq \frac{1}{32(2000L^2+1)}\cdot d_{\mathbb G}(p,q).
\end{equation*}
Therefore, $|H(p)-H(q)| \geq C(L)\cdot d_{\mathbb G}(p,q)$.

If $p \in Q,\; q\in Q^{**}$, then there is a $R \in Q^*$ so that $p,\;q \in R^*$ and $\widetilde h_{R^*}$ is bi-Lipschitz on $R^*$. Then, we can use same argument as the first case to conclude
\begin{equation*}
 |H(p)-H(q)| \geq |\widetilde h_{R^*}(p)-\widetilde h_{R^*}(q)| \geq \frac{1}{C_1}d_{\mathbb G}(p,q).
\end{equation*}
\end{proof}

\subsection{Global bi-Lipschitz embedding on the Grushin plane}
Finally, let us consider the map 
\begin{equation}
F(p)=g(p) \times H(p) \times \dist(p, A)\; \text{for}\; p\in \Omega. 
\end{equation}
Then, $F$ is  Lipschitz on $\Omega$ because $g$ and  $\text{dist}(\cdotp, A)$ are Lipschitz and $H$ is a finite sum of Lipschitz maps. Also, $F$ is  co-Lipschitz on $\Omega$ by Lemma 3.2 and Lemma 3.6. Therefore, $F$ is a bi-Lipschitz embedding on its dense set $\overline{ \Omega}=\mathbb G$ into $\mathbb {R}^3 \times\mathbb {R}^{2M} \times \mathbb {R}$. Moreover, the bi-Lipschitz constant depends only on the bi-Lipschitz constant $L$ of $f$

\begin{remark}
The dimension $2M+4$ of the Euclidean space depends only on the bi-Lipschitz constant $L$ of the bi-Lipschitz embedding $f$ of the singular line. However, the number of colors $M$ in Lemma 3.4 is somewhat large. The question, what is the minimal dimension of Euclidean space into which the Grushin plane bi-Lipschitzly embeds, remains open. 
\end{remark}

\bibliographystyle{plain}
\bibliography{BL_tex_Mac.bib}
\end{document}